\documentclass[final]{siamltex}

\usepackage{graphicx,amsmath,amsfonts,cite}
\usepackage[title,titletoc,toc]{appendix}

\newtheorem{prop}{Proposition}
\newtheorem{teo}{Theorem}
\newtheorem{rem}{Remark}
\newtheorem{alg}{Algorithm}
\newtheorem{lem}{Lemma}

\newcommand{\xx}{\boldsymbol}
\newcommand{\yy}{\displaystyle}

\def\M3as{Mathematical Models and Methods in Applied Sciences}

\title{On the stability of a loosely-coupled scheme based on a Robin interface condition for fluid-structure interaction\thanks{C. Vergara has been partially
supported by the H2020-MSCA-ITN-2017, EU project 765374 "ROMSOC - Reduced Order
Modelling, Simulation and Optimization of Coupled systems''.}}

\author{Giacomo Gigante\thanks{Dipartimento di Ingegneria Gestionale, dell'Informazione e della Produzione, 
Universit\`a degli Studi di Bergamo, Italy, ({\tt giacomo.gigante@unibg.it})}
\and Christian Vergara\thanks{LABS, Dipartimento di Chimica, Materiali e Ingegneria Chimica "Giulio Natta", Politecnico di Milano, Italy, ({\tt christian.vergara@polimi.it})}
}

\begin{document}

\maketitle

\begin{abstract}
We consider a loosely coupled algorithm for fluid-structure interaction based on a Robin
interface condition for the fluid problem (explicit Robin-Neumann scheme).
We study the dependence of the stability of this method on the interface parameter in the Robin condition.
In particular, for a model problem we find sufficient conditions for instability and stability of the method.
In the latter case, we found a stability condition relating the time discretization parameter, 
the interface parameter, and the added mass effect. Numerical experiments confirm the theoretical findings
and highlight optimal choices of the interface parameter that guarantee an accurate solution
with respect to an implicit one.
\end{abstract}



\pagestyle{myheadings}
\thispagestyle{plain}
\markboth{G. GIGANTE, C. VERGARA}{STABILITY OF A LOOSELY-COUPLED SCHEME FOR FSI}

\section{Introduction}
\label{sec:intro}
Loosely-coupled schemes (also known as {\sl explicit}) are a very attractive strategy for the numerical solution of the fluid-structure
interaction (FSI) problem. Indeed, they are based on the solution of just one fluid and one structure problem
at each time step, thus allowing a big improvement in the computational costs in comparison to fully-coupled ({\sl implicit}) partitioned procedures 
and monolithic schemes. Another interesting feature of such schemes is that pre-existing fluid and structure
solvers could be often employed.  

For these reasons, loosely-coupled schemes have been widely used in many engineering applications
such as aeroelasticity \cite{parkf1,pipernof1,farhatv1}. However, the stability properties of such schemes
deteriorate when the so-called {\sl added mass effect} becomes relevant. This happens, in particular, when the fluid and structure densities
are comparable, as happens in hemodynamics \cite{quarteronim1}. For example, in \cite{causing1} it has been proven that 
the classical explicit {\sl Dirichlet-Neumann} scheme is unconditionally unstable in the hemodynamic regime, see also \cite{forsterw1,nobilev2}. 

In the recent years, there has been a growing interest in partitioned procedures that are based on
{\sl Robin interface conditions}. The latter are obtained by considering linear combinations of the standard interface conditions owing to the introduction
of suitable parameters. The choice of such parameters is crucial for accelerating the convergence of implicit schemes
\cite{badian1,badian2,astorinoc1,gerardon1,gigantev1}.
Some works focused then on the design of stable loosely-coupled schemes for large added mass effect,
which are based on Robin interface conditions
\cite{nobilev1,guidobonig1,fernandez1,fernandezm3,bukacc1,bukacc2,lukacovar1,banksh1,burmanf5}.
These studies proposed specific values of the interface parameters which guarantee good stability properties
(possibly in combination with suitable stabilizations). 

In this paper, the explicit {\sl Robin-Neumann} scheme, obtained by equipping the fluid subproblem with a Robin condition with parameter $\alpha$ and
the structure one with a Neumann condition, is considered. In particular, it is investigated how the choice of the interface parameter $\alpha$
influences the stability of the numerical solution. To this aim, two analyses on a simplified problem are performed, the first one determining sufficient conditions for instability of the scheme,
whereas the second one sufficient conditions for its stability.
This will allow us to understand the dependence of stability and instability on the physical and numerical parameters and to 
properly design stable loosely-coupled schemes which could be easily implemented also by means of available (even commercial) solvers. 

To validate the theoretical findings found in the analysis reported in Section 2, 
in Section 3 we eventually present the results of some numerical experiments,
where the issue of accuracy is also discussed by proposing some ''optimal''
value of the interface parameters.

\section{Position of the problem}

\subsection{The fluid-structure interaction problem}
We introduce in what follows the simplified fluid-structure interaction problem considered in the analyses below, see \cite{causing1}.
We consider the 2D fluid domain $\Omega^f$ which is a rectangle $R\times L$, where $R$ is denoted the ``radius'' and $L$ the length of
the domain. Its boundary is given by $\partial\Omega^f=\Gamma\bigcup\Sigma$, where $\Sigma$ is the part where the interaction 
with the structure occurs, see Figure \ref{fig:domain}.
For the structure, we consider a 1D model in the domain $\Omega^s=\Sigma$. 
\begin{figure}[!h]
\centering
\includegraphics[width=8cm,height=3.5cm]{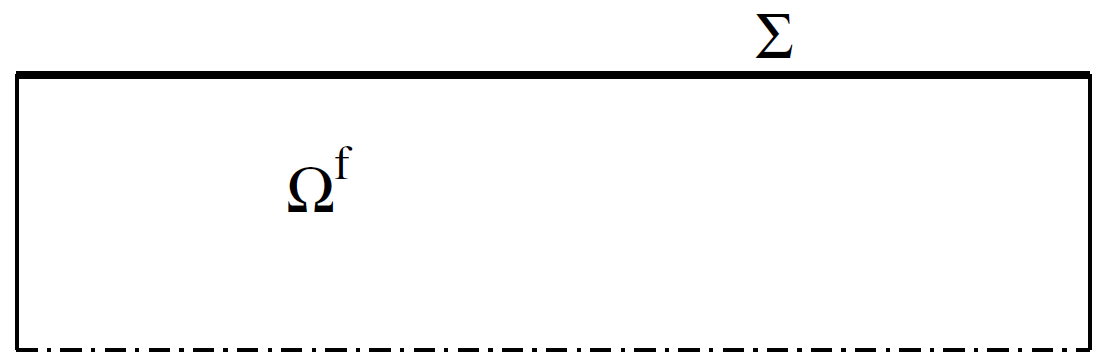}
\caption{Fluid and structure domains for the simplified fluid-structure interaction problem.}
\label{fig:domain}
\end{figure}

For the fluid modeling, we consider a linear incompressible inviscid problem, whereas for the structure the independent rings model \cite{quarteronit1}.
The displacement could happen only in the radial direction.
Moreover, we assume small displacements so that the structure deformation is negligible 
and the fluid domain can be considered fixed.
Thus, we have the following FSI problem:\\
Find fluid velocity $\xx u$, fluid pressure $p$, and structure displacement $\eta$, such that 
\begin{subequations}
\label{eq:fsi}
\begin{align}
& \yy\rho_f\frac{\partial\xx u}{\partial t} + \nabla p = 0 & \text{in}\,\, (0,T)\times\Omega^f,\\
& \nabla\cdot\xx u = 0 & \text{in}\,\, (0,T)\times\Omega^f,\\
&\yy\xx u\cdot\xx n = \frac{\partial\eta}{\partial t} & \text{in}\,\, (0,T)\times\Sigma,\label{eq:fsi-3}\\
&\yy\rho_sH_s\frac{\partial^2\eta}{\partial t^2} + \beta\eta - \psi\frac{\partial^2\eta}{\partial x^2} = p & \text{in}\,\, (0,T)\times\Sigma,\label{eq:fsi-4}
\end{align}
\end{subequations}
where $\xx n$ is the outward normal, $\rho_f$ and $\rho_s$ the fluid and structure densities, $x$ the axial direction along which $\Sigma$ is located, 
$H_s$ the structure thickness, and $\beta$ and $\psi$ two suitable parameters accounting for the elasticity of the structure. 
Moreover, we have to equip the fluid problem with boundary conditions (for example of homogeneous type) on $\partial\Omega\setminus\Sigma$ and for the tangential component on $\Sigma$ \cite{causing1}.
Condition $\eqref{eq:fsi-3}$ represents a no-slip condition at the interface $\Sigma$ between fluid and structure
({\sl perfect adherence} or {\sl kinematic} condition). Due to the lower space dimension of the structure, the independent rings model \eqref{eq:fsi-4} 
represents also the third Newton law (continuity of the normal stresses or {\sl dynamic} condition).


\subsection{Time discretization and explcit Robin-Neumann scheme}
Denoting by $\Delta t$ the time discretization parameter, $v^n$ the approximation of
$v(t^n),\,t^n=n\Delta t$, and setting $u^n=(\xx u^n\cdot\xx n)|_{\Sigma}$, we have the following discretized-in-time version of problem \eqref{eq:fsi}:
Find for each $n$, $\xx u^n,\,p^n$, and $\eta^n$, such that 
\begin{subequations}
\label{eq:fsi-discr}
\begin{align}
& \yy\rho_f\frac{\xx u^n-\xx u^{n-1}}{\Delta t} + \nabla p^n = 0 & \text{in}\,\,\Omega^f,\label{eq:fsi-discr-1}\\
& \nabla\cdot\xx u^n = 0 & \text{in}\,\, \Omega^f,\label{eq:fsi-discr-2}\\
&\yy u^n = \frac{\eta^n-\eta^{n-1}}{\Delta t} & \text{on}\,\, \Sigma,\label{eq:fsi-discr-3}\\
&\yy\rho_sH_s\delta_{tt}\eta^{n} + \beta\eta^n - \psi\frac{\partial^2\eta^n}{\partial x^2} = p^n & \text{in}\,\,\Sigma.\label{eq:fsi-discr-4}
\end{align}
\end{subequations}
Notice that we have considered a backward Euler approximation for the fluid problem 
and we indicated with $\delta_{tt}$ the approximation of the second derivative in time for the structure
problem. 

Introducing $\alpha> 0$, we can substitute in \eqref{eq:fsi-discr} the kinematic condition \eqref{eq:fsi-discr-3}
with the following linear combination obtained with the dynamic condition \eqref{eq:fsi-discr-4}:
\begin{equation}
\label{eq:fsi-discr-3b}
- \alpha u^n + p^n = - \yy \alpha\frac{\eta^n-\eta^{n-1}}{\Delta t} + \rho_sH_s\delta_{tt}\eta^{n}
+ \beta\eta^n - \psi\frac{\partial^2\eta^n}{\partial x^2} \quad \text{on}\,\,\Sigma.
\end{equation}
Of course, the solution of problem \eqref{eq:fsi-discr-1}-\eqref{eq:fsi-discr-2}-\eqref{eq:fsi-discr-3b}-\eqref{eq:fsi-discr-4} coincides with that
of \eqref{eq:fsi-discr}.

Now, the idea is to consider a partitioned method where condition \eqref{eq:fsi-discr-3b} is given 
to the fluid problem, whereas \eqref{eq:fsi-discr-4} is in fact the structure problem. 
Since the fluid problem has been discretized with an implicit method, in \eqref{eq:fsi-discr-3b}
we use the following approximation of the second derivative
\[
\yy\delta_{tt}\eta^n = \rho_sH_s\frac{\eta^{n}-2\eta^{n-1}+\eta^{n-2}}{\Delta t^2}. 
\]
Instead, for the structure problem \eqref{eq:fsi-discr-4} we use the explicit leap-frog approximation
\[
\yy\delta_{tt}\eta^n = \rho_sH_s\frac{\eta^{n+1}-2\eta^{n}+\eta^{n-1}}{\Delta t^2}. 
\]
We observe that, due to the explicit time discretization of the structure problem, fluid and structure are in fact decoupled and, accordingly, we can introduce the following algorithm.\\
\newpage
\begin{alg}
{\bf Explicit Robin-Neumann algorithm.}
\label{alg1}
Given $\xx u^0,\,\eta^1,\,\eta^0$, for $n\ge 1$, at time step $t^n$:
\begin{enumerate}
\item Solve the fluid problem with a Robin condition at the interface $\Sigma$:
\begin{subequations}
\label{eq:fsi-discr-alg}
\begin{align}
& \yy\rho_f\frac{\xx u^n-\xx u^{n-1}}{\Delta t} + \nabla p^n = 0 & \text{in}\,\,\Omega^f,\\
& \nabla\cdot\xx u^n = 0 & \text{in}\,\, \Omega^f,\\
&\yy- \alpha u^n + p^n = - \yy \alpha\frac{\eta^n-\eta^{n-1}}{\Delta t} + \rho_sH_s\frac{\eta^{n}-2\eta^{n-1}+\eta^{n-2}}{\Delta t^2} + \beta\eta^n - \psi\frac{\partial^2\eta^n}{\partial x^2} & \text{on}\,\, \Sigma;
\label{eq:fsi-discr-alg-3}
\end{align}
\end{subequations}

\item Then, solve the structure problem (Neumann
condition at the interface):
\begin{equation}
\label{eq:fsi-discr-alg2}
\yy\rho_sH_s\frac{\eta^{n+1}-2\eta^{n}+\eta^{n-1}}{\Delta t^2} + \beta\eta^n - \psi\frac{\partial^2\eta^n}{\partial x^2} = p^n \quad  \text{in}\,\,\Sigma.
\end{equation}

\end{enumerate}
\end{alg}


In the next section, we study how the stability of the previous algorithm is affected by the choice of the  parameter $\alpha$.

\section{Stability analysis}

\subsection{Preliminaries}
First, we notice that using the previous algorithm, the discrete kinematic condition \eqref{eq:fsi-discr-3}
is not satisfied anymore. Indeed, from \eqref{eq:fsi-discr-alg2} we have
\[
\yy\beta\eta^n - \psi\frac{\partial^2\eta^n}{\partial x^2} - p^n
= - \rho_sH_s\frac{\eta^{n+1}-2\eta^{n}+\eta^{n-1}}{\Delta t^2}, 
\]
where it is understood from now on that the equalities we derive hold true at the interface $\Sigma$.
By introducing the latter expression in \eqref{eq:fsi-discr-alg-3}, we obtain
\[
\yy- \alpha u^n = - \yy \alpha\frac{\eta^n-\eta^{n-1}}{\Delta t} + \rho_sH_s\frac{\eta^{n}-2\eta^{n-1}+\eta^{n-2}}{\Delta t^2} - \rho_sH_s\frac{\eta^{n+1}-2\eta^{n}+\eta^{n-1}}{\Delta t^2}, 
\]
which leads to 
\begin{equation}
\label{eq:kinematic}
\yy u^n = \frac{\eta^n-\eta^{n-1}}{\Delta t} + \rho_sH_s\frac{\eta^{n+1}-3\eta^{n}+3\eta^{n-1}-\eta^{n-2}}{\alpha\Delta t^2}.
\end{equation}
The latter equality provides a ''correction'' of the discrete kinematic condition \eqref{eq:fsi-discr-3}
as a consequence of the explicit treatment.

Following then \cite{causing1}, we consider the {\sl added mass operator} $\mathcal M:H^{-1/2}(\Sigma)\to H^{1/2}(\Sigma)$, 
which allows us to write the following relation between fluid pressure and velocity
at the interface, under the assumption of null external pressure:
\begin{equation}
\label{eq:added}
\yy p=-\rho_f\mathcal M\left(\frac{\partial(\xx u\cdot\xx n)}{\partial t}\right)\quad \text{in}\,\, (0,T)\times\Sigma.
\end{equation}
At the time discrete level, we have
\[
\yy p^n =-\rho_f\mathcal M\left(\frac{u^n-u^{n-1}}{\Delta t}\right).
\]
Inserting \eqref{eq:kinematic} for both $u^n$ and $u^{n-1}$, we obtain
\begin{equation}
\label{eq:added-mass}
\yy p^n =-\rho_f\mathcal M\left(\frac{\eta^{n}-2\eta^{n-1}+\eta^{n-2}}{\Delta t^2}
+ \frac{\rho_sH_s}{\alpha\Delta t^3}\left(\eta^{n+1}-4\eta^{n}+6\eta^{n-1}-4\eta^{n-2}+\eta^{n-3}\right)
\right),
\end{equation}
which gives a relation between pressure and displacement at the interface.

We can write $\eta^m$ for any $m$ as a linear combination of the $L^2$ orthonormal basis functions
$\left\{g_i(x)=\sqrt{2/L}\sin\left(\frac{i\pi x}{L}\right)\right\}$:
\[
\eta^m(x) = \sum_{i=1}^\infty\eta^m_ig_i(x),
\]
for suitable coefficients $\eta^m_i$, see \cite{causing1,badian1}. Notice that 
$g_i$ are eigenfunctions of both the added mass operator $\mathcal M$
and of the Laplace operator $\mathcal L=-b\,\partial_{xx}|_\Sigma$,
with eigenvalues given by 
\[
\yy\mu_i = \frac{L}{i\pi\tanh\left(\frac{i\pi R}{L}\right)},\quad \yy\lambda_i = \left(\frac{i\pi}{L}\right)^2,
\]
respectively.

It is useful for later purposes to highlight that the eigenvalues $\mu_i$ and $\lambda_i$ of the discrete versions of the operators 
$\mathcal M$ and $\mathcal L$ obtained with finite elements, feature the following properties \cite{causing1,badian1}:
\begin{equation}
\label{eq:eigen}
\yy \mu_{min}  \sim h,\qquad \lambda_{max} \sim h^{-2},\qquad\mu_{max}\sim h^0,
\end{equation}
where $h$ is the space discretization parameter.

\subsection{Sufficient conditions for instability}
We present in what follows a first result that provides sufficient conditions
that guarantee conditional instability of the explicit Robin-Neumann scheme. This results generalizes
the one proven in \cite{causing1} about the unconditional instability of the Dirichlet-Neumann scheme
($\alpha\to +\infty$, see Proposition 3 in \cite{causing1}).

\begin{prop}
\label{prop:suff}
The explicit Robin-Neumann scheme is unstable if 
\begin{equation}
\label{eq:unst2}
\rho_{s}H_{s} < \max_{i}\gamma_i,\quad \gamma_i=\alpha\Delta t\frac{4\rho_{f}\mu_{i}+\Delta t^{2}\left(
\beta+\psi\lambda_{i}\right)  }{16\rho_{f}\mu_{i}+4\alpha\Delta t}.
\end{equation}
\end{prop}

\begin{proof}

We start by inserting in the interface condition \eqref{eq:fsi-discr-alg2}
the expression of $p^n$ given by \eqref{eq:added-mass}, obtaining
\[
\yy\rho_sH_s\frac{\eta^{n+1}-2\eta^{n}+\eta^{n-1}}{\Delta t^2} + \beta\eta^n - \psi\frac{\partial^2\eta^n}{\partial x^2} +
\]
\[
\yy\rho_f\mathcal M\left(\frac{\eta^{n}-2\eta^{n-1}+\eta^{n-2}}{\Delta t^2}
+ \frac{\rho_sH_s}{\alpha\Delta t^3}\left(\eta^{n+1}-4\eta^{n}+6\eta^{n-1}-4\eta^{n-2}+\eta^{n-3}\right)\right)=0.
\]
Notice that the previous is a relation in the discrete structure displacement solely.
Multiplying it by the basis function $g_i$ and integrating over the interface $\Sigma$,
we obtain 
\begin{align*}
& \yy \frac{\rho_sH_s}{\Delta t^2}\left (1+\frac{\rho_f\mu_i}{\alpha\Delta t}\right)\eta_i^{n+1}
+ \left(-\frac{2\rho_sH_s}{\Delta t^2} + \beta + \psi\lambda_i + \frac{\rho_f\mu_i}{\Delta t^2} 
- \frac{4\rho_sH_s\rho_f\mu_i}{\alpha\Delta t^3}\right)\eta_i^n + \\
& \left( \frac{\rho_s H_s}{\Delta t^2} - 2\frac{\rho_f\mu_i}{\Delta t^2} 
+ \frac{6\rho_sH_s\rho_f\mu_i}{\alpha\Delta t^3} \right)\eta_i^{n-1}
+ \left( \frac{\rho_f\mu_i}{\Delta t^2} - \frac{4\rho_sH_s\rho_f\mu_i}{\alpha\Delta t^3}\right)\eta_i^{n-2}
+ \frac{\rho_sH_s\rho_f\mu_i}{\alpha\Delta t^3}\eta_i^{n-3} = 0.
\end{align*}
By multiplying the last identity by $\frac{\alpha\Delta t}{\rho_f\mu_i}$,
we obtain the following characteristic polynomial corresponding to the previous difference equation:
\begin{equation}
\label{eq:chi}
\left.
\begin{array}{ll}
\chi(y)  = & \yy \frac{\rho_sH_s}{\Delta t^2}\left (1+\frac{\alpha\Delta t}{\rho_f\mu_i}\right)y^4
+ \left(-\frac{2\alpha\rho_sH_s}{\rho_f\mu_i\Delta t} 
+ \frac{\alpha\Delta t}{\rho_f\mu_i}\left(\beta + \psi\lambda_i\right) + \frac{\alpha}{\Delta t} 
- \frac{4\rho_sH_s}{\Delta t^2}\right)y^3 \\
& \yy + \left( \frac{\alpha\rho_s H_s}{\rho_f\mu_i\Delta t} - 2\frac{\alpha}{\Delta t} 
+ \frac{6\rho_sH_s}{\Delta t^2} \right)y^2
+ \left( \frac{\alpha}{\Delta t} - \frac{4\rho_sH_s}{\Delta t^2}\right) y
+ \frac{\rho_sH_s}{\Delta t^2}.
\end{array}
\right.
\end{equation}
Now, we compute the value of the previous polynomial for $y=-1$:
\begin{align*}
\chi(-1) & =  \frac{16\rho_sH_s}{\Delta t^2} + \frac{4\alpha\rho_s H_s}{\rho_f\mu_i\Delta t}
- \frac{4\alpha}{\Delta t} -\frac{\alpha\Delta t}{\rho_f\mu_i}\left(\beta + \psi\lambda_i\right)\\
 & = \frac{\alpha}{\rho_f\mu_i\Delta t}\left(	4\rho_sH_s - 4\rho_f\mu_i -\Delta t^2 \left(\beta + \psi\lambda_i\right) \right) + \frac{16\rho_sH_s}{\Delta t^2}.
\end{align*}
It follows that $\chi(-1)<0$ under condition \eqref{eq:unst2}.
Since $\lim_{y\to-\infty}\chi(y)=+\infty$, it follows that in this case there exists
at least one real root $\bar y<-1$ of the polynomial associated to the difference equation, 
implying that the method is unstable.  
\end{proof}

\subsection{Sufficient conditions for stability}
We discuss in the following result some sufficient conditions that guarantee
that the explicit Robin-Neumann scheme is conditionally stable. The idea is to start again from 
the polynomial \eqref{eq:chi} and discuss when 
its four roots have all modulus less than 1. 

To this aim, we first introduce
the following version of the implicit function theorem.

\begin{teo}
\label{Dini}Let $f\in\mathcal{C}^{1}\left(  \mathbb{R}^{2}\right)  $ and
suppose that for all $x\in\Omega$, an open interval, and for all
\[
y\in\left(  \varphi_{1}\left(  x\right)  ,\varphi_{2}\left(  x\right)
\right),
\]
where $\varphi_{1},\varphi_{2}:\Omega\rightarrow\mathbb{R}$ are continuous
functions, we have either
\begin{equation}
\label{eq:dini0}
\frac{\partial f}{\partial y}\left(  x,y\right)  \geq b\left(  x\right)  >0
\end{equation}
or%
\[
\frac{\partial f}{\partial y}\left(  x,y\right)  \leq b\left(  x\right)  <0,
\]
for some continuous function $b:\Omega\rightarrow\mathbb{R}$. Let
$g:\Omega\rightarrow\mathbb{R}$ be such that for all $x\in\Omega$%
\begin{subequations}
\begin{align}
g\left(  x\right)   &  \in\left(  \varphi_{1}\left(  x\right)  ,\varphi
_{2}\left(  x\right)  \right), \label{eq:dini1}\\
g\left(  x\right)  -\frac{f\left(  x,g\left(  x\right)  \right)  }{b\left(
x\right)  } &  \in\left(  \varphi_{1}\left(  x\right)  ,\varphi_{2}\left(
x\right)  \right)  .\label{eq:dini2}
\end{align}
\end{subequations}
Then, there is a unique function $\xi:\Omega\rightarrow\mathbb{R}$ such that, for
all $x\in\Omega$, $\xi\left(  x\right)  \in\left(  \varphi_{1}\left(  x\right)
,\varphi_{2}\left(  x\right)  \right)  $ and $f\left(  x,\xi\left(  x\right)
\right)  =0.$ Furthermore, for all $x\in\Omega$
\[
\left\vert \xi\left(  x\right)  -g\left(  x\right)  \right\vert \leq\left\vert
\frac{f\left(  x,g\left(  x\right)  \right)  }{b\left(  x\right)  }\right\vert
.
\]

\end{teo}

\begin{proof}
Let us consider the case $\frac{\partial f}{\partial y}\left(  x,y\right)
\geq b\left(  x\right)  >0$. The other case follows from it after replacing
$f\left(  x,y\right)  $ with $-f\left(  x,y\right)  $ and $b\left(  x\right)
$ with $-b\left(  x\right)  $. Fix $x\in\Omega$. By strict monotonicity of the
function $y\mapsto f\left(  x,y\right)  $ in the interval $\left(  \varphi
_{1}\left(  x\right)  ,\varphi_{2}\left(  x\right)  \right)  $, there is at
most one value $\xi\left(  x\right)  \in\left(  \varphi_{1}\left(  x\right)
,\varphi_{2}\left(  x\right)  \right)  $ for which $f\left(  x,\xi\left(
x\right)  \right)  =0$. This proves uniqueness. 

Now, the function $y\mapsto
f\left(  x,y\right)  $ takes the value $f\left(  x,g\left(  x\right)  \right)
$ at $y_{0}=g\left(  x\right)  $. Consider now the point
\[
y_{1}=y_{0}-\frac{f\left(  x,y_{0}\right)  }{b\left(  x\right)  },
\]
and assume without loss of generality that $y_{0}\leq y_{1}$, that is
$f\left(  x,y_{0}\right)  \leq0$. By the hypotheses, we have $\left[
y_{0},y_{1}\right]  \subseteq\left(  \varphi_{1}\left(  x\right)  ,\varphi
_{2}\left(  x\right)  \right)  $, so that for all $y\in\left[  y_{0}%
,y_{1}\right]  $ we have $\frac{\partial f}{\partial y}\left(  x,y\right)
\geq b\left(  x\right)  >0$ and%
\begin{align*}
f\left(  x,y_{1}\right)   &  =\int_{y_{0}}^{y_{1}}\frac{\partial f}{\partial
y}\left(  x,t\right)  dt+f\left(  x,y_{0}\right)  \\
&  \geq b\left(  x\right)  \left(  y_{1}-y_{0}\right)  +f\left(
x,y_{0}\right)  \\
&  =b\left(  x\right)  \left(  y_{0}-\frac{f\left(  x,y_{0}\right)  }{b\left(
x\right)  }-y_{0}\right)  +f\left(  x,y_{0}\right)  =0.
\end{align*}
By the intermediate value theorem, there is a point $\xi\left(  x\right)
\in\left[  y_{0},y_{1}\right]  $ such that $f\left(  x,\xi\left(  x\right)
\right)  =0$. This prove existence. 

Finally,
\[
y_{0}\leq \xi\left(  x\right)  \leq y_{1}%
\]
means%
\[
g\left(  x\right)  \leq \xi\left(  x\right)  \leq g\left(  x\right)
-\frac{f\left(  x,g\left(  x\right)  \right)  }{b\left(  x\right)  }%
\]
so that
\[
\left\vert \xi\left(  x\right)  -g\left(  x\right)  \right\vert \leq\left\vert
\frac{f\left(  x,g\left(  x\right)  \right)  }{b\left(  x\right)  }\right\vert,
\]
proving the last part of the theorem.
\end{proof}
\\
\\
Now, we are ready to introduce the main result of this work. 

\begin{teo}
\label{t1}
For all positive $\rho_{s}$, $\rho_{f}$, $\alpha$, $H_{s}$ and for all finite
sequences of positive eigenvalues $\left\{  \mu_{i}\right\}  _{i=1}^{N}$,
$\left\{  \lambda_{i}\right\}  _{i=1}^{N}$, there exists a positive $\delta$
such that, if $0<\Delta t<\delta$, then, for all $i=1,\ldots,N$, the polynomial
\eqref{eq:chi}
has four simple roots in the open unit disc in the complex plane.
\end{teo}

\begin{proof}
It is convenient to normalize $\chi\left(  y\right) $, setting
\begin{align*}
Q\left(  y\right)   &  =\frac{\Delta t^{2}}{\rho_{s}H_{s}}\chi\left(
y\right)  \\
&  =\left(  1+B_{i}z\right)  y^{4}-\left(  4+\left(  2B_{i}-A\right)
z-AC_{i}z^{3}\right)  y^{3}+\left(  6+\left(  B_{i}-2A\right)  z\right)
y^{2}-\left(  4-Az\right)  y+1,
\end{align*}
where%
\begin{align*}
z &  =\Delta t>0,\\
A &  =\frac{\alpha}{\rho_{s}H_{s}}>0,\\
B_{i} &  =\frac{\alpha}{\rho_{f}\mu_{i}}>0,\\
C_{i} &  =\frac{1}{\rho_{f}\mu_{i}}\left(  \beta+\psi\lambda_{i}\right)  >0.
\end{align*}
For technical reasons, we will equivalently study the polynomial%
\begin{equation}
\left.
\label{eq:P}
\begin{array}{lll}
\yy P\left(  x\right)   &  = & \yy x^{4}Q\left(  \frac{1}{x}\right)  \\
&  = & x^{4}-\left(  4-Az\right)  x^{3}+\left(  6+\left(  B_{i}-2A\right)
z\right)  x^{2}\\
& & -\left(  4+\left(  2B_{i}-A\right)  z-AC_{i}z^{3}\right)
x+\left(  1+B_{i}z\right),
\end{array}
\right.
\end{equation}
showing that all its four roots have modulus greater than $1$. Observe first
that when $z=0$, $P\left(  x\right)  $ reduces to
\[
x^{4}-4x^{3}+6x^{2}-4x+1=\left(  x-1\right)  ^{4}.
\]
By classical results, this implies that for sufficiently small $z$, 
$P\left(  x\right)  $ has four simple roots as close as desired
to $x=1$ in the complex plane (see \cite{knopp1}, page 122). Unfortunately, this is not sufficient
for our purposes, and for this reason we need a deeper analysis. 

Now set $x=1+U$. This simplifies our formulas, since we look for roots that
are close to $1$. This gives
\begin{equation}
\label{eq:P2}
P\left(  1+U\right)  =U^{4}+AzU^{3}+\left(  B_{i}+A\right)  zU^{2}+AC_{i}%
z^{3}U+AC_{i}z^{3}.
\end{equation}
Recalling that $U$ is a complex variable, we write it as $U=u+iv$ where
$u$ and $v$ are its real and imaginary parts, respectively. Thus, the equation%
\[
\left(  u+iv\right)  ^{4}+Az\left(  u+iv\right)  ^{3}+\left(  B_{i}+A\right)
z\left(  u+iv\right)  ^{2}+AC_{i}z^{3}\left(  u+iv\right)  +AC_{i}z^{3}=0
\]
reduces to the system%
\begin{equation}
\label{eq:system}
\left\{
\begin{array}
[c]{l}%
v^{4}-v^{2}\left(  6u^{2}+3Azu+z\left(  A+B_{i}\right)  \right)
+u^{4}+Azu^{3}+\left(  A+B_{i}\right)  zu^{2} \\
\qquad\qquad+AC_{i}z^{3}u+AC_{i}z^{3}=0,\\
v\left(  -v^{2}\left(  4u+Az\right)  +4u^{3}+3Au^{2}z+2uz\left(
A+B_{i}\right)  +AC_{i}z^{3}\right)  =0.
\end{array}
\right.
\end{equation}
Notice that the solution of the second equation $v=0$ reduces the first
equation to%
\[
u^{4}+\left(  u+1\right)  Azu^{2}+B_{i}zu^{2}+AC_{i}z^{3}\left(  u+1\right)
=0,
\]
which does not have any real solution $u>-1$ (all summands are positive).
Since we look for roots $u+iv$ close to $0$, we disregard the solution of the
second equation $v=0$ and focus on the other solution
\begin{equation}
\label{eq:v2}
v^{2}=\frac{4u^{3}+3Au^{2}z+2uz\left(  A+B_{i}\right)  +AC_{i}z^{3}}{4u+Az},
\end{equation}
which reduces the first equation to
\begin{equation}
T_{6}u^{6}+T_{5}u^{5}+T_{4}u^{4}+T_{3}u^{3}+T_{2}u^{2}+T_{1}u+T_{0}%
=0,
\label{equation_sixth}%
\end{equation}
where%
\begin{align*}
T_{6} &  =-64,\\
T_{5} &  =-96Az,\\
T_{4} &  =32\left(  A+B_{i}\right)  z+48A^{2}z^{2},\\
T_{3} &  =32A\left(  A+B_{i}\right)  z^{2}+8A^{3}z^{3},\\
T_{2} &  =4\left(  A+B_{i}\right)  ^{2}z^{2}+8A\left(  A^{2}+B_{i}%
A-2C_{i}\right)  z^{3}+4A^{2}C_{i}z^{4},\\
T_{1} &  =-2A\left(  A+B_{i}\right)  ^{2}z^{3}+8A^{2}C_{i}z^{4}-2A^{3}%
C_{i}z^{5},\\
T_{0} &  =-A^{2}C_{i}\left(  B_{i}-zC_{i}\right)  z^{5}.
\end{align*}

Since we look for roots that go to $0$ with $z$, we may assume that
$u$ is $O\left(  z\right)  $ as $z\rightarrow0$. The original
equation \eqref{equation_sixth} can therefore be approximated with
\[
T_{2}u^{2}+T_{1}u+T_{0}=0;
\]
further, disregarding all higher order terms in $z$, it can be approximated with
\begin{equation}
\label{eq:teo}
4\left(  A+B_{i}\right)  ^{2}u^{2}-2A\left(  A+B_{i}\right)  ^{2}zu-A^{2}%
z^{3}C_{i}B_{i}=0.
\end{equation}
If $u=O\left(  z\right)  $, then the third term in the above equation 
can be neglected and the same equation 
can be approximated with%
\[
4u-2Az=0,
\]
which gives the approximate solution
\[
u=\frac{A}{2}z.
\]
If $u=O\left(  z^{2}\right)  $, then the first term in equation \eqref{eq:teo} can be neglected
and \eqref{eq:teo} can be approximated
with%
\[
2\left(  A+B_{i}\right)  ^{2}u+Az^{2}C_{i}B_{i}=0,
\]
which gives the approximate solution
\[
u=-\frac{AC_{i}B_{i}}{2\left(  A+B_{i}\right)  ^{2}}z^{2}.
\]

We now need to estimate the derivative with respect to $u$ of
\[
f\left(  z,u\right)  =T_{6}u^{6}+T_{5}u^{5}+T_{4}u^{4}+T_{3}u^{3}+T_{2}%
u^{2}+T_{1}u+T_{0},
\]
that is
\begin{equation}
\label{eq:teo2}
\frac{\partial f}{\partial u}\left(  z,u\right)  =6T_{6}u^{5}+5T_{5}%
u^{4}+4T_{4}u^{3}+3T_{3}u^{2}+2T_{2}u+T_{1}.
\end{equation}

We are now ready to apply Theorem \ref{Dini}, with the two following choices
for $g$ suggested by the previous approximate solutions:%
\begin{align*}
g_{1}(z) &  =\frac{A}{2}z,\\
g_{2}(z) &  =-\frac{AC_{i}B_{i}}{2\left(  A+B_{i}\right)  ^{2}}z^{2}.
\end{align*}

Assume first $u$ close to $g_{1}\left(  z\right)  .$ Precisely, assume that
for some $K_{1}>0$, $\varphi_{1}\left(  z\right)  =g_{1}\left(  z\right)
-K_{1}z^{2}\leq u\leq\varphi_{2}\left(  z\right)  =g_{1}\left(  z\right)
+K_{1}z^{2}$, so that hypothesis \eqref{eq:dini1} in Theorem \ref{Dini}
is satisfied. In particular, $u=g_{1}\left(  z\right)  +O\left(  z^{2}\right)
$ and, from \eqref{eq:teo2},
\begin{align*}
\frac{\partial f}{\partial u}\left(  z,u\right)   &  =8\left(  A+B_{i}\right)
^{2}z^{2}u-2A\left(  A+B_{i}\right)  ^{2}z^{3}+O\left(  z^{4}\right)  \\
&  =8\left(  A+B_{i}\right)  ^{2}z^{2}g_{1}\left(  z\right)  -2A\left(
A+B_{i}\right)  ^{2}z^{3}+O\left(  z^{4}\right)  \\
&  =2A\left(  A+B_{i}\right)  ^{2}z^{3}+O\left(  z^{4}\right)  .
\end{align*}
Thus, for some small positive $\delta_1,$ if $0<z<\delta_1$ and $g_{1}\left(
z\right)  -K_{1}z^{2}\leq u\leq g_{1}\left(  z\right)  +K_{1}z^{2}$, we have
\[
\frac{\partial f}{\partial u}\left(  z,u\right)  \geq A\left(  A+B_{i}\right)
^{2}z^{3}=b_{1}\left(  z\right),
\]
which shows that, with the above choices, hypothesis \eqref{eq:dini0} in Theorem \ref{Dini}
is satisfied.

Now we need to check hypothesis \eqref{eq:dini2} in Theorem \ref{Dini}, that is 
\[
g_{1}\left(  z\right)  -K_{1}z^{2}\leq g_{1}\left(  z\right)  -\frac{f\left(
z,g_{1}\left(  z\right)  \right)  }{b_{1}\left(  z\right)  }\leq g_{1}\left(
z\right)  +K_{1}z^{2},%
\]
or, in other words,
\[
\frac{f\left(  z,g_{1}\left(  z\right)  \right)  }{b_{1}\left(  z\right)
}=O\left(  z^{2}\right)  .
\]
We have:
\begin{align*}
\frac{f\left(  z,g_{1}\left(  z\right)  \right)  }{b_{1}\left(  z\right)  } &
=\frac{4\left(  A+B_{i}\right)  ^{2}z^{2}g_{1}\left(  z\right)  ^{2}-2A\left(
A+B_{i}\right)  ^{2}z^{3}g_{1}\left(  z\right)  +O\left(  z^{5}\right)
}{A\left(  A+B_{i}\right)  ^{2}z^{3}}\\
&  =\frac{\left(  A+B_{i}\right)  ^{2}z^{2}A^{2}z^{2}-A\left(  A+B_{i}\right)
^{2}z^{3}Az+O\left(  z^{5}\right)  }{A\left(  A+B_{i}\right)  ^{2}z^{3}}\\
&  =O\left(  z^{2}\right),
\end{align*}
since the first two terms at numerator vanish.
Thus, it follows from Theorem \ref{Dini} that there exists a unique function
\[
u_1:\left(  0,\delta_1\right)  \rightarrow\left(  \frac{A}{2}z-K_{1}%
z^{2},\frac{A}{2}z+K_{1}z^{2}\right)
\]
such that $f\left(  z,u_1\left(z\right)  \right)  =0$ for all $z\in\left(0,\delta_1\right).$ Furthermore,
for all $z\in\left(  0,\delta_1\right)  $%
\[
\left\vert u_1\left(  z\right)  -\frac{A}{2}z\right\vert \leq\frac{f\left(
z,g_{1}\left(z\right)  \right)  }{b_{1}\left(z\right)  },%
\]
which implies that
\[
u_1\left(z\right)  =\frac{A}{2}z+O\left(  z^{2}\right)
\]
as $z\rightarrow 0^+.$

Similarly, let us assume now $u$ close to $g_{2}\left(  z\right)  .$ Precisely,
assume that for some $K_{2}>0$, $\psi_{1}\left(  z\right)  =g_{2}\left(
z\right)  -K_{2}z^{3}\leq u\leq\psi_{2}\left(  z\right)  =g_{2}\left(
z\right)  +K_{1}z^{3}$. Then, we have $u=g_2(z) + O\left(  z^{3}\right)$ and 
\[
\frac{\partial f}{\partial u}\left(  z,u\right)  =-2A\left(  A+B_{i}\right)
^{2}z^{3}+O\left(  z^{4}\right)  .
\]
Thus, for some small positive $\delta_2,$ if $0<z<\delta_2$ and $g_{2}\left(
z\right)  -K_{2}z^{3}\leq u\leq g_{2}\left(  z\right)  +K_{2}z^{3}$, we have
\[
\frac{\partial f}{\partial u}\left(  z,u\right)  \leq-A\left(  A+B_{i}\right)
^{2}z^{3}=b_{2}\left(  z\right)  .
\]
Now we need to check if
\[
g_{2}\left(  z\right)  -K_{2}z^{3}\leq g_{2}\left(  z\right)  -\frac{f\left(
z,g_{2}\left(  z\right)  \right)  }{b_{2}\left(  z\right)  }\leq g_{2}\left(
z\right)  +K_{2}z^{3}%
\]
or, in other words, if
\[
\frac{f\left(  z,g_{2}\left(  z\right)  \right)  }{b_{2}\left(  z\right)
}=O\left(  z^{3}\right)  .
\]
We have:
\begin{align*}
\frac{f\left(  z,g_{2}\left(  z\right)  \right)  }{b_{2}\left(  z\right)  } &
=\frac{-2A\left(  A+B_{i}\right)  ^{2}z^{3}g_{2}\left(  z\right)  -A^{2}%
z^{5}C_{i}B_{i}+O\left(  z^{6}\right)  }{-A\left(  A+B_{i}\right)  ^{2}z^{3}%
}\\
&  =\frac{z^{3}A^{2}C_{i}B_{i}z^{2}-A^{2}z^{5}C_{i}B_{i}+O\left(
z^{6}\right)  }{A\left(  A+B_{i}\right)  ^{2}z^{3}}\\
&  =O\left(  z^{3}\right)  .
\end{align*}
Thus, it follows from Theorem \ref{Dini} that there is a unique function
\[
u_2:\left(  0,\delta_2\right)  \rightarrow\left(  -\frac{AC_{i}B_{i}%
}{2\left(  A+B_{i}\right)  ^{2}}z^{2}-K_{2}z^{3},-\frac{AC_{i}B_{i}}{2\left(
A+B_{i}\right)  ^{2}}z^{2}+K_{2}z^{3}\right)
\]
 such that $f\left(
z,u_2\left(  z\right)  \right)  =0$ for all $z\in\left(  0,\delta_2\right)  .$
Furthermore, for all $z\in\left(  0,\delta_2\right)  $%
\[
\left\vert u_2\left(  z\right)  +\frac{AC_{i}B_{i}}{2\left(  A+B_{i}\right)
^{2}}z^{2}\right\vert \leq\frac{f\left(  z,g_{2}\left(  z\right)  \right)
}{b_{2}\left(  z\right)  },%
\]
which implies that
\[
u_2\left(  z\right)  =-\frac{AC_{i}B_{i}}{2\left(  A+B_{i}\right)  ^{2}%
}z^{2}+O\left(  z^{3}\right)
\]
as $z\rightarrow 0^+.$

Let us now compute $v^{2}$ by means of \eqref{eq:v2} for these two choices. 
Setting $u=u_1\left(z\right)  =\frac{A}{2}z+O\left(  z^{2}\right)  $, we obtain%
\begin{align*}
v_{1}^{2} &  =\frac{4u_{1}^{3}+3Au_{1}^{2}z+2u_{1}z\left(  A+B_{i}\right)
+AC_{i}z^{3}}{4u_{1}+Az}\\
&  =\frac{Az^{2}\left(  A+B_{i}\right)  +O\left(  z^{3}\right)  }{3Az+O\left(
z^{2}\right)  }\\
&  =\frac{\left(  A+B_{i}\right)  }{3}z+O\left(  z^{2}\right).
\end{align*}
Setting now $u=u_2\left(z\right)  =-\frac{AC_{i}B_{i}%
}{2\left(  A+B_{i}\right)  ^{2}}z^{2}+O\left(  z^{3}\right)  $, we obtain%
\begin{align*}
v_{2}^{2} &  =\frac{4u_{2}^{3}+3Au_{2}^{2}z+2u_{2}z\left(  A+B_{i}\right)
+AC_{i}z^{3}}{4u_{2}+Az}\\
&  =\frac{2z\left(  A+B_{i}\right)  \left(  -\frac{AC_{i}B_{i}}{2\left(
A+B_{i}\right)  ^{2}}z^{2}\right)  +AC_{i}z^{3}+O\left(  z^{4}\right)
}{Az+O\left(  z^{2}\right)  }\\
&  =\frac{\frac{AC_{i}A}{A+B_{i}}z^{3}+O\left(  z^{4}\right)  }{Az+O\left(
z^{2}\right)  }\\
&  =\frac{C_{i}A}{A+B_{i}}z^{2}+O\left(  z^{3}\right).
\end{align*}

Thus, $(u_1,v_1)$ and $(u_2,v_2)$ are solutions of system \eqref{eq:system}. 
This means that $u_1\pm iv_1$ and $u_2\pm iv_2$ are the 4 roots of \eqref{eq:P2}
and $1+u_1\pm iv_1$ and $1+u_2\pm iv_2$ are the four roots of \eqref{eq:P}.

The final step of the proof is to show that these four roots for $P(x)$ have modulus
greater than 1. Accordingly, we have
\[
\left(  1+u_{1}\right)  ^{2}+v_{1}^{2}=1+Az+\frac{\left(  A+B_{i}\right)  }%
{3}z+O\left(  z^{2}\right), \quad\text{as $z\rightarrow 0^+$}.
\]
Thus, there exists a positive $\delta_3$ such that $\left(  1+u_{1}\right)  ^{2}+v_{1}^{2}>1$
for $z<\delta_3$. Moreover, we have
\begin{align*}
\left(  1+u_{2}\right)  ^{2}+v_{2}^{2} &  =\left(  1-\frac{AC_{i}B_{i}%
}{2\left(  A+B_{i}\right)  ^{2}}z^{2}\right)  ^{2}+\frac{C_{i}A}{\left(
A+B_{i}\right)  }z^{2}+O\left(  z^{3}\right)  \\
&  =1-\frac{AC_{i}B_{i}}{\left(  A+B_{i}\right)  ^{2}}z^{2}+\frac{C_{i}%
A}{\left(  A+B_{i}\right)  }z^{2}+O\left(  z^{3}\right)  \\
&  =1+\frac{C_{i}A^{2}}{\left(  A+B_{i}\right)  ^{2}}z^{2}+O\left(
z^{3}\right) ,\quad\text{as $z\rightarrow 0^+$}.
\end{align*}
Thus, there exists a positive $\delta_4$ such that $\left(  1+u_{2}\right)  ^{2}+v_{2}^{2}>1$
for $z<\delta_4$.

This proves that the four roots of $\chi(y)$ in \eqref{eq:chi} have all
modulus strictly less than 1, provided that $\Delta t<\delta=\min\{\delta_1,\delta_2,\delta_3,\delta_4\}$.

\end{proof}

\subsection{Discussion}\label{sec:disc}
The results proven above give us sufficient conditions for instability 
or stability of the explicit Robin-Neumann scheme. In particular, in the last case, we found that,
given $\alpha>0$, for $\Delta t$ small enough this scheme is absolute stable. 

In view of the numerical experiments, 
we discuss hereafter more in detail the hypotheses of the previous results.
To this aim, in what follows we write three conditions that imply \eqref{eq:unst2}: 
\begin{itemize}
\item[i)]
\begin{equation}
\rho_{s}H_{s}<\alpha\Delta t\frac{4\rho_{f}\mu_{\min}+\Delta t^{2}\left(
\beta+\psi\lambda_{\max}\right)  }{16\rho_{f}\mu_{\min}+4\alpha\Delta
t}=\overline{\eta}.\label{instability_two}%
\end{equation}
This is obtained by taking the greatest possible value of $i$ in (\ref{eq:unst2});

\item[ii)]
\begin{equation}
\left\{
\begin{array}
[c]{l}%
\rho_{s}H_{s}<\rho_{f}\mu_{\min}+\Delta t^{2}\left(  \beta+\psi\lambda_{\max
}\right)  /4=\eta_{1}\\
\\
\alpha>\dfrac{16\rho_{f}\mu_{\min}\rho_{s}H_{s}}{\Delta t\left(  4\rho_{f}%
\mu_{\min}+\Delta t^{2}\left(  \beta+\psi\lambda_{\max}\right)  -4\rho
_{s}H_{s}\right)  }=\alpha_{1}.
\end{array}
\right.  \label{instability_two_bis}%
\end{equation}
This is obtained by solving (\ref{instability_two}) in the variable $\alpha$;

\item[iii)]
\begin{equation}
\left\{
\begin{array}
[c]{l}%
\rho_{s}H_{s}<\rho_{f}\mu_{1}=\eta_{2}\\
\\
\alpha>\dfrac{4\rho_{f}\mu_{1}\rho_{s}H_{s}}{\Delta t\left(  \rho_{f}\mu
_{1}-\rho_{s}H_{s}\right)  }=\alpha_{2}.
\end{array}
\right.  \label{instability_three_bis}%
\end{equation}
This is obtained by taking $i=1$ in (\ref{eq:unst2}), deleting the term $\Delta t^{2}\left(
\beta+\psi\lambda_{1}\right)  $ and solving in $\alpha$.
\end{itemize}

\begin{enumerate}
\item {\bf Dependence on $\rho_{s}H_{s}$}. 
By looking at conditions
(\ref{instability_two_bis}) and (\ref{instability_three_bis}), we see that
when $\rho_{s}H_{s}<\max\left(  \eta_{1},\eta_{2}\right)  $, the explicit
Robin-Neumann scheme is unstable if $\alpha$ is large enough. In
particular, for decreasing values of $\rho_{s}H_{s}<\eta_{1}$ (i.e. for
increasing added mass effect), the value of $\alpha_{1}$ decreases, enlarging
the range of $\alpha$ such that the scheme is unstable.
The same argument holds true for $\alpha_2$ when $\rho_{s}H_{s}<\eta_{2}$.

\item {\bf Dependence on $\Delta t$}.
From Theorem \ref{t1}, we have that for any fixed
$\alpha>0$, the explicit Robin-Neumann scheme is stable provided that $\Delta
t$ is small enough. This result is consistent with Proposition 1.
Indeed, for all indices $i$, we have
\[
\lim_{\Delta t\rightarrow0} \gamma_i = 
\lim_{\Delta t\rightarrow0} \alpha\Delta t\frac{4\rho_{f}\mu_{i}+\Delta
t^{2}\left(  \beta+\psi\lambda_{i}\right)  }{16\rho_{f}\mu_{i}+4\alpha\Delta
t}=0.
\]
Accordingly, we have also from \eqref{instability_two_bis}
\[
\lim_{\Delta t\rightarrow0} \alpha_1 = +\infty.
\]

Observe also that, by (\ref{instability_three_bis}), for $\rho_{s}H_{s}<\eta_{2}$
instability of the scheme follows if
\[
\Delta t>\dfrac{4\rho_{f}\mu_{1}\rho_{s}H_{s}}{\alpha\left(  \rho_{f}\mu
_{1}-\rho_{s}H_{s}\right)  }.
\]
This means that in order to have stability according to Theorem \ref{t1} we need at
least $\delta\leq c\alpha^{-1}$. The dependence of $\delta$ on $\alpha$ is
still under study and will be hopefully discussed in future studies.

\item {\bf Dependence on $h$}. By exploiting the behaviour of $\mu_{min}$ and 
$\lambda_{max}$ with respect to $h$, see \eqref{eq:eigen}, 
we find from \eqref{instability_two} that $\lim
_{h\rightarrow0}\overline{\eta}=+\infty$ for $\Delta t$ fixed.
This means that the stability properties of the method deteriorates when 
$h$ decreases. This is justified by the fact that when $h$ becomes small,
the fluid and structure solutions should match a large number of d.o.f.
at the interface
and, due to the implicit treatment for the fluid time discretization and the explicit one for the structure, 
this matching becomes more and more difficult 
for smaller values of $h$.

\item {\bf $\xx\Delta \xx t,\xx h\to \xx 0$}.
From \eqref{eq:eigen}, we have that
if $\Delta t\sim h$, then   
$\lim_{\Delta t,h\to 0}\alpha_1\simeq
\frac{16\rho_sH_s\rho_f}{\psi
+ 4\rho_f-4\rho_sH_s}$.
When the added mass effect is large enough ($\rho_sH_s$ small enough), this limit is 
positive and bounded, unlike the case $\Delta t\to 0, h$ fixed (cf. point 2).
This means that in this case we have instability for a wide range of values of $\alpha$
even as $\Delta t\to 0$. This suggests that 
the value of $\delta$ in Theorem \ref{t1} should be smaller (up to a constant) than $h$.
Indeed, we have the following result.
\begin{lem}
The value of $\delta$ in Theorem \ref{t1} satisfies the relation
\[
\delta < ch,
\]
for a suitable constant $c$.
\end{lem}
\begin{proof}
Fix $\alpha>0$. Call $\delta=f(h)$ the relationship between $\delta$ and $h$. 
Let assume that the thesis is not true, that is that there exists
a sequence of values $h_j$ and correspondingly $\delta_j=f(h_j)$, such that 
\begin{equation}
\label{eq:lemma}
\yy\lim_{j\to+\infty}\frac{\delta_j}{h_j} = +\infty.
\end{equation}
From Theorem \ref{t1}, we have that stability
is guaranteed if $\Delta t=\Delta t_j=\delta_j/2$.
On the other side, from \eqref{eq:lemma} and the choice of $\Delta t$ above, we have that
$h_j=o(\delta_j)=o(\Delta t_j)$.
From \eqref{instability_two}, we have that in the range
$h=o\left(  \Delta t\right)  $, $\lim_{\Delta t,h\rightarrow0}\overline
{\eta}=+\infty$, obtaining unconditional instability. 
This contradicts the previous finding about stability. This means that the thesis is true.
\end{proof}

\item {\bf The cases $\alpha=+\infty$ and $\alpha=0$}. 
Theorem \ref{t1} holds true for any $\alpha\in(0,+\infty)$. The case $\alpha=+\infty$ corresponds
to the explicit Dirichlet-Neumann scheme. In this case, the polynomial $\chi(y)$ in \eqref{eq:chi}
corresponds to the one found in \cite{causing1} (see Proposition 3 therein), where it is shown that at least one root
has modulus greater than one independently of $\Delta t$
({cf. also \eqref{instability_two}}). 

Regarding $\alpha=0$, we obtain $\chi(y)=\frac{\rho_{s}H_{s}}{\Delta t^{2}}(y-1)^4$.
This means that the solution does not blow up, even if it is not strictly absolute stable.
This is in accordance with the fact that in this case the
numerical solution does not evolve during the time evolution, being always equal to the initial condition
(the same Neumann datum is transferred at the interface).
Thus, accuracy is completely lost.
From this observation, we can argue that
too small values of $\alpha$, even though give stability, 
do not lead to an accurate solution.

\end{enumerate}

\section{Numerical results}

\subsection{Preliminaries}
In this section we present some numerical results with the aim of validating the 
theoretical findings reported in the previous section. 
In particular, we studied the effectiveness of the 
analyses obtained for the simplified models, 
when applied to complete three-dimensional fluid and structure models. 
All the simulations are inspired from 
hemodynamics. This problem is of great interest for our purposes
since it is characterized by a high added mass effect, so that 
the stability of explicit methods is a challenging issue. Moreover, the thickness
of the structure is small enough to make acceptable the use of a membrane model
in the stability analysis of the previous section.    

We considered the coupling between the 3D incompressible Navier-Stokes equations written in 
the Arbitrary Lagrangian-Eulerian formulation \cite{donea2} and the 3D linear infinitesimal elasticity,
see \cite{nobilep1,gigantev1} for all the details. 
For the time discretization, we used the BDF scheme of order 1 for both the 
fluid and structure, 
with a semi-implicit treatment of the fluid convective term, whereas
for the numerical discretization
we employed $P1bubble-P1$ finite elements for the fluid and $P1$ finite elements for the structure.
The fluid domain at each time step is obtained by extrapolation of previous time steps 
(semi-implicit approach \cite{fernandezg2,badiaq1,nobilep2}).
We used the following data: 
fluid viscosity $\mu=0.035\,g/(cm\,s)$, fluid density $\rho_f=1\,g/cm^3$,
structure density $\rho_s=1.1\,g/cm^3$, Poisson ratio $\nu=0.49$, Young modulus $E=3\cdot 10^5\,Pa$,
surrounding tissue parameter for the structure problem
\cite{moireaux1} $\gamma_{ST}=1.5\cdot 10^5\,Pa/cm$.

The fluid domain is a cylinder with length $L=5\,cm$ and radius $R=0.5\,cm$,
whereas the structure domain is the 
external cylindrical crown with thickness $H_s=0.1\,cm$. We considered
a couple of meshes with
4680 tetrahedra and 1050 vertices for the fluid and 1260 vertices for the structure (mesh I).
Another couple of meshes (mesh II) has been obtained by halving the values of the space
discretization parameter.

At the inlet we prescribed a Neumann condition given by the following pressure function
\[
P_{in} = 
\yy 500\left(1-\cos\left(\frac{2\pi t}{0.01}\right)\right)\,\,dyne/cm^2,\quad t\le T=0.04\,s,\\ 
\] 
with absorbing resistance conditions at the outlets \cite{nobilev1, nobilep1}.

All the numerical results have been obtained with the parallel Finite Element library \verb|LIFEV| \cite{lifev}.

\subsection{On the stability of the numerical solution}
In the first set of numerical experiments, we study the stability of the solution.
The time discretization parameter is $\Delta t = 0.0005\,s$. In Figure \ref{fig:press} we report the mean pressure over the middle cross section of the cylinder
obtained for different values of $\alpha$ and with an implicit method.  
\begin{figure}[!h]
\centering
\includegraphics[width=14cm,height=8.cm]{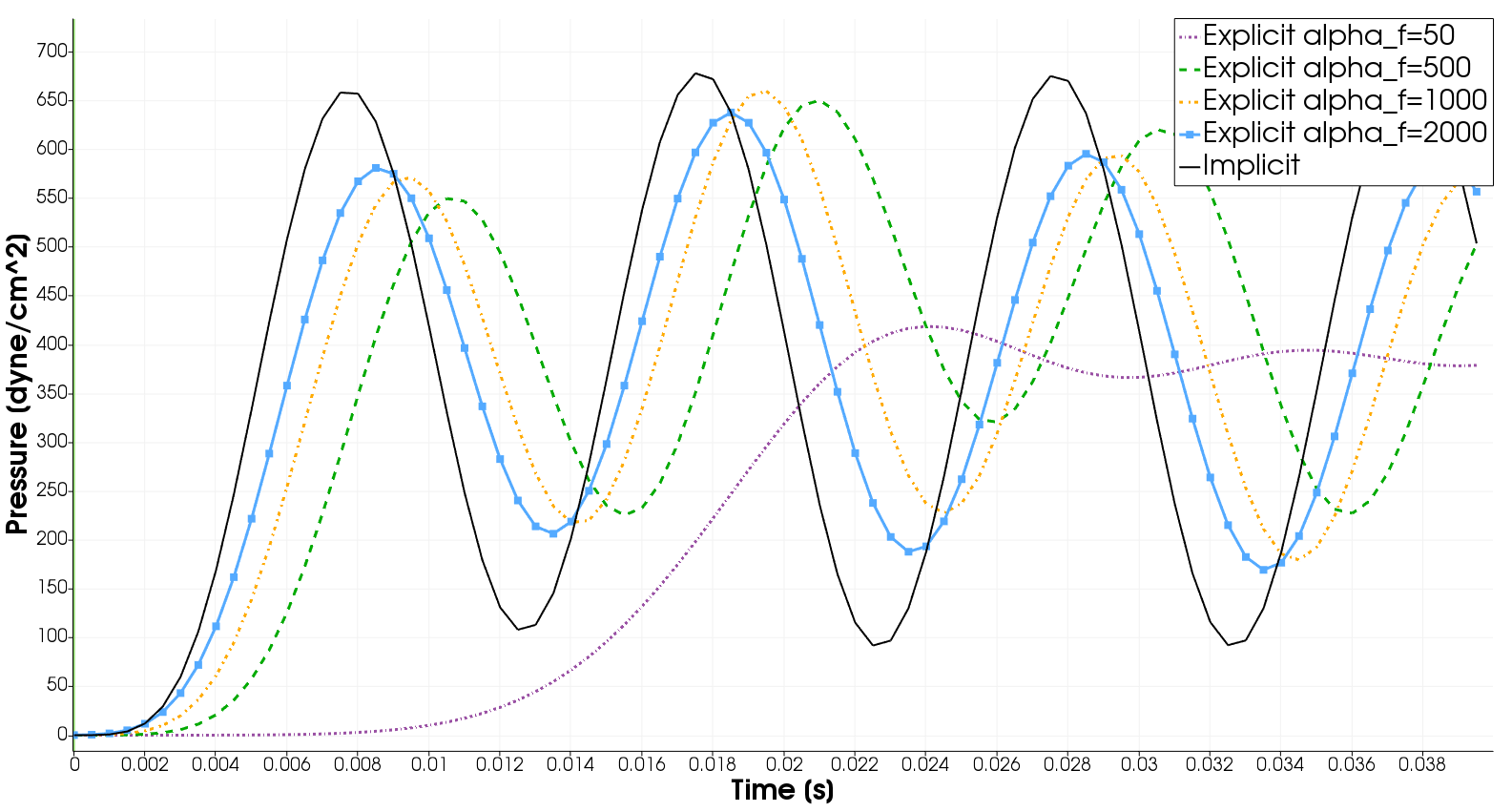}
\caption{Fluid mean pressure over the middle cross section ($z=L/2=2.5\,cm$) for different values
of $\alpha$.}
\label{fig:press}
\end{figure}
From these results we observe stability of the numerical solution obtained with the explicit Robin-Neumann method
for some values of the interface parameter $\alpha$.
The accuracy seems to deteriorate for decreasing values of $\alpha$ (cf. point 5 in the Discussion
of Sect. \ref{sec:disc}). 
Notice that with $\alpha=2500$ the numerical solution (not reported here) blows up.
The same happens for bigger values of $\alpha$.
This is consistent with the result proven in Proposition 1, 
for which an unstable solution is obtained for $\alpha$ greater than a threshold
when the added mass effect is large enough,
see \eqref{instability_two_bis}-\eqref{instability_three_bis}
(cf. also point 1 in the Discussion
of Sect. \ref{sec:disc}).

In Table \ref{tab1} we indicate if stability is achieved for different 
space and time discretizations parameters. 
\begin{table}[!h]
\begin{center}
\begin{tabular}{ccc|c}
Mesh & $\Delta t$ & $\alpha$ & Stability \\
\hline
I & $0.625\cdot 10^{-4}$ & 4689 & OK \\
I & $1.25\cdot 10^{-4}$ & 4689 & NO \\
I & $1.25\cdot 10^{-4}$ & 2000 & OK \\
I & $5\cdot 10^{-4}$ & 2000 & OK \\
II & $5\cdot 10^{-4}$ & 2000 & NO \\
\end{tabular}
\caption{Stability of the explicit RN scheme for different values of the parameters.}
\label{tab1}
\end{center}
\end{table}
From the first two rows, we observe that, given a value of $\alpha$, stability is achieved 
only if $\Delta t$ is small enough. This was expected from the theoretical findings, see 
Theorem \ref{t1}.
From the first three lines, we observe that the value of $\Delta t$ needed to have stability
could be increased when $\alpha$ is decreased. This is in accordance with  
point 2 in the Discussion of Sect. \ref{sec:disc}. Finally, from the last two rows, we find that stability
is achieved for given values of $\alpha$ and $\Delta t$ if the mesh is not too fine, thus confirming the observation
made in point 3 of the Discussion of Sect. \ref{sec:disc}. 

\subsection{On the accuracy of the numerical solution}
Two very interesting topics that are not yet discussed are: 
i) how to select not
empirically a reasonable value of $\alpha$ that should fall in the range of stability and
ii) which is the dependence of the accuracy of explicit Robin-Neumann methods on $\alpha$.
In what follows we provide some preliminary hints about these two points.

The idea we propose is to use an optimal value of $\alpha$ which should guarantee
efficient convergence of the implicit Robin-Neumann scheme. Since such a value
makes the convergence factor of the method (thus the error at each iteraton) small, 
we expect that the use of the same value
of $\alpha$ for the explicit counterpart of the method should reduce the error accumulated at each 
time step. 

In particular, we propose here to use the optimality result proven in \cite{gigantev1}
for the implicit Robin-Robin method in the case of cylindrical interfaces. This leads to two optimal values $\alpha^{opt}_f$
and $\alpha^{opt}_s$ in the Robin interface conditions. Since in the hemodynamic regime 
the convergence properties of the
implicit Robin-Neumann scheme with the optimal value $\alpha^{opt}_f$ are very similar to that of the
''optimal'' implicit Robin-Robin scheme \cite{gerardon1}, 
we propose here to use $\alpha=\alpha^{opt}_f$ as an effective value that should provide
a stable and accurate solution for the explicit Robin-Neumann scheme.

We consider the same numerical experiment as above, 
with time discretization parameter $\Delta t = 0.001\,s$. 
In Figure \ref{fig:press2}, we report the mean pressure over the middle cross section 
for $\alpha=\alpha_f^{opt}=2250$ and $\alpha=1500, 3000$. 
\begin{figure}[!h]
\centering
\includegraphics[width=14cm,height=8.cm]{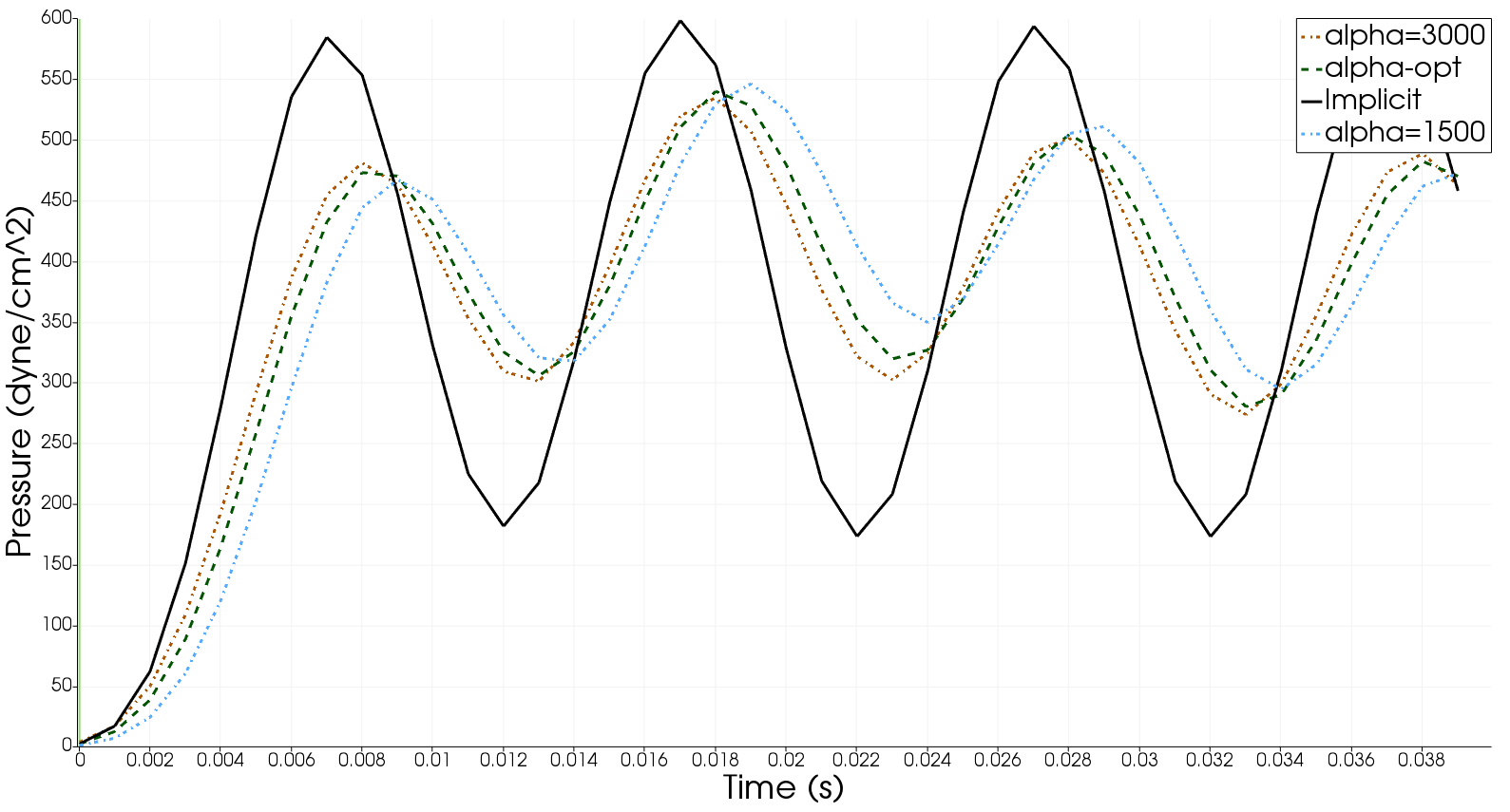}
\caption{Fluid mean pressure over the middle cross section ($z=L/2=2.5\,cm$) for different values
of $\alpha$ close to $\alpha_f^{opt}$.}
\label{fig:press2}
\end{figure}

For $\alpha>3500$
the solution blows up, whereas for $\alpha<1500$ the accuracy deteriorates. We observe that the solution obtained with the value $\alpha_f^{opt}$ proposed a priori is very
close to the optimal one ($\alpha=3000$) found empirically. This result highlights that an effective choice
of $\alpha$ that guarantees stability and accuracy is possible, at least in the case of a cylindrical domain.

\begin{rem}
Notice that, from empirical observations running the implicit Robin-Robin scheme, 
we found that $\alpha_f^{opt}$ is greater
than zero and does not blow up when $\Delta t\to 0$. Thus, 
from point 2 of the Discussion of Sect. \ref{sec:disc}, we have that $\delta\le c(\alpha^{opt})^{-1}$ 
still makes sense also
when $\Delta t\to 0$, since $\delta>0$. 
\end{rem}

\subsection{Conclusions}
In this work we have proposed an analysis of stability of a loosely coupled scheme 
of Robin-Neumann type for fluid-structure interaction, possibly featuring a large added mass effect.
In order to make the results found in this paper reliable for concrete applications, 
we need now to apply them to realistic geometries and boundary data. 
This would make the explicit Robin-Neumann scheme an effective strategy to be considered
for example in hemodynamics, where the added mass effect is elevated. 
This is actually outside the aims of the paper which is focused on the analysis of sufficient conditions for conditional stability and instability. For this reason, we are currently studying this topic for a future development 
of this work. What in our opinion is promising is that our numerical experiments 
in the hemodynamic regime highlighted
an excellent agreement with the theoretical findings and lead to accurate solutions in the range of stability.

\bibliographystyle{siam}
\bibliography{gv3-SINUM}

\end{document}